\documentclass{amsart}
\usepackage{amsmath,amsthm,amssymb,amsfonts,amscd}
\usepackage[small,nohug]{diagrams}
\usepackage[francais]{ babel}
\usepackage{tikz}

\usepackage{amssymb,latexsym}
\usepackage{color}

\usepackage{amsthm}
\usepackage{eucal}

\title[Spectre et topologie des vari\'et\'es hyperboliques]{Quelques cons\'equences des travaux d'Arthur pour le spectre et la topologie des
vari\'et\'es hyperboliques}

\author{Nicolas Bergeron et Laurent Clozel} 
\address{Institut de Math\'ematiques de Jussieu \\
Unit\'e Mixte de Recherche 7586 du CNRS \\
Universit\'e Pierre et Marie Curie \\
4, place Jussieu 75252 Paris Cedex 05, France \\}

\email{bergeron@math.jussieu.fr}
\urladdr{http://people.math.jussieu.fr/~bergeron}

\address{Universit\'e Paris Sud \\
Unit\'e Mixte de Recherche 8628 du CNRS \\
Laboratoire de Math\'ematiques \\
B\^at. 425, 91405 Orsay cedex, France\\}

\email{Laurent.Clozel@math.u-psud.fr}

 \DeclareFontFamily{OT1}{rsfs}{}
 
\DeclareFontShape{OT1}{rsfs}{n}{it}{<-> rsfs10}{}
\DeclareMathAlphabet{\mathscr}{OT1}{rsfs}{n}{it}

\renewcommand{\H}{\mathbf{H}}

\DeclareFontFamily{OT1}{rsfs}{}

\DeclareFontShape{OT1}{rsfs}{n}{it}{<-> rsfs10}{}
\DeclareMathAlphabet{\mathscr}{OT1}{rsfs}{n}{it}

\newcommand{\SO}{\mathrm{SO}}

\swapnumbers
\newtheorem{thm}[subsection]{Th\'eor\`eme}  
         
\newtheorem*{lem*}{Lemme}         
\newtheorem{prop}[subsection]{Proposition}
\newtheorem*{prop*}{Proposition}

\newtheorem{conj}[subsection]{Conjecture}

\newtheorem{cor}[subsection]{Corollaire}

\theoremstyle{definition}

\numberwithin{equation}{subsection}

\renewcommand{\H}{\mathbf H}  

\newcommand{\cal}{\mathcal}
\newcommand{\GL}{\mathrm{GL}}
\newcommand{\SL}{\mathrm{SL}}

\newcommand{\SU}{\mathrm{SU}}

\newcommand{\Sp}{\mathrm{Sp}}

\newcommand{\RR}{\mathbf R}
\newcommand{\CC}{\mathbf C}
\newcommand{\ZZ}{\mathbf Z}
\newcommand{\QQ}{\mathbf Q}

\def\adots{\mathinner{\mkern2mu\raise1pt\hbox{.}
\mkern3mu\raise4pt\hbox{.}\mkern1mu\raise7pt\hbox{.}}}

\begin{document}

\begin{abstract}  
En nous basant sur les r\'esultats d'Arthur annonc\'es dans \cite[\S 30]{Arthur} 
nous d\'emontrons
les conjectures \'enonc\'ees dans \cite{IMRN,BC,SMF} dans le cas des groupes orthogonaux \`a l'exclusion des groupes de type ${}^6 \- D_4$. En ce qui concerne ces derniers, nous annon\c{c}ons 
la d\'emonstration -- encore en pr\'eparation~-- que leur r\'eseaux de congruences ont toujours
un $H^1$ trivial. Les d\'emonstrations d'Arthur devraient para\^{\i}tre prochainement.
\end{abstract}
\maketitle
\tableofcontents

%

\section{Introduction}

Soit $G$ le $\QQ$-groupe semi-simple obtenu -- par restriction des scalaires -- \`a partir d'un groupe 
sp\'ecial orthogonal sur un corps de nombres totalement r\'eel. On consid\`ere pour l'instant 
un groupe $G$ qui ne provient pas d'une forme tordue ${}^{3,6} \- D_4$ -- on dit alors que $G$
est {\it non trialitaire}. On
suppose enfin $G (\RR)$ \'egal au produit de $\SO (n,1)$ par un groupe compact. 
L'espace sym\'etrique associ\'e au groupe $G$ est alors l'espace hyperbolique r\'eel 
$\H^n$ que l'on munit de sa m\'etrique de courbure sectionnelle constante \'egale \`a $-1$. 
\'Etant donn\'e un sous-groupe de congruence (sans torsion) $\Gamma \subset G$ on peut
former le quotient $\Gamma \backslash \H^n$. Ce quotient 
est une vari\'et\'e hyperbolique r\'eelle de volume fini; appelons {\it vari\'et\'es hyperboliques de congruence} les vari\'et\'es ainsi obtenues. Le cas \'ech\'eant nous parlons de 
vari\'et\'es hyperboliques de congruence {\it non trialitaires} afin de rappeler que le groupe 
$G$ est suppos\'e non trialitaire.

Burger, Li et Sarnak \cite{BLS} ont propos\'e la conjecture suivante -- dite {\it de puret\'e} ou 
{\it de quantification} -- d\'ecrivant le spectre du laplacien sur les vari\'et\'es hyperboliques de congruence. 

\begin{conj} \label{CBS}
Soit $M$ une vari\'et\'e hyperbolique de congruence de dimension $n$. Le spectre $L^2$
du laplacien (sur les fonctions) de $M$ est contenu dans l'ensemble
$$\bigcup_{0 \leq j < \frac{n-1}{2}} \left\{ \left( \frac{n-1}{2} \right)^2 - \left(  \frac{n-1}{2} - j \right)^2 \right\}
\cup \left[ \left( \frac{n-1}{2} \right)^2 , + \infty \right[ .$$
\end{conj}

Un cas particulier -- le cas du groupe $\SO (3)$ d\'eploy\'e $/ \QQ$ -- de cette conjecture est la c\'el\`ebre conjecture de Selberg. Dans \cite[\S 6.2 \& 6.3]{BC} nous avons ramen\'e cette conjecture aux conjectures d'Arthur telles que formul\'ees dans \cite{ArthurOU} et propos\'e des extensions au spectre
du laplacien sur les formes diff\'erentielles. Depuis, le programme d'Arthur a fait de spectaculaires 
progr\`es, d'une part gr\^ace \`a Arthur lui-m\^eme -- voir \cite[\S 30]{Arthur} -- et d'autre part gr\^ace \`a la
d\'emonstration du lemme fondamental -- par Ng\^o \cite{Ngo} -- puis de ses diff\'erents avatars (tordus et
pond\'er\'es) par Ng\^o, Waldspurger, Laumon et Chaudouard.  \`A la suite de quoi les r\'esultats contenus dans \cite[\S 30]{Arthur} peuvent maintenant \^etre rendus inconditionnels. C'est la raison d'\^etre de cet article.
Nous commen\c{c}ons par d\'eduire de ces bouleversements r\'ecents le th\'eor\`eme \ref{Tprinc} qui implique en particulier l'approximation suivante de la conjecture \ref{CBS}~:

\begin{thm} 
Soit $M$ une vari\'et\'e hyperbolique de congruence non trialitaire de dimension $n$.  Le spectre $L^2$
du laplacien (sur les fonctions) de $M$ est contenu dans l'ensemble
$$\bigcup_{0 \leq j < \frac{n-1}{2}} \left\{ \left( \frac{n-1}{2} \right)^2 - \left(  \frac{n-1}{2} - j \right)^2 \right\}
\cup \left[ \left( \frac{n-1}{2} \right)^2 - \left( \frac12 - \frac{1}{(n+1)^2+1} \right)^2 , + \infty \right[ .$$
\end{thm}

Un aspect remarquable de ce r\'esultat est que, bien que l'on ne connaisse qu'une approximation de la
\og conjecture de Ramanujan \fg \ donnant la d\'ecomposition spectrale des formes automorphes pour
$\GL (n)$, le spectre -- au moins dans l'intervalle $[0 , \left( \frac{n-2}{2} \right)^2 ]$ -- est exactement celui
pr\'edit par Burger et Sarnak. 

Notons -- cela d\'ecoule de \cite{BLS} -- que les valeurs propres \og pures \fg \ apparaissent bien dans le spectre de certaines vari\'et\'es hyperboliques de congruence.

Le th\'eor\`eme \ref{Tprinc} implique plus g\'en\'eralement un r\'esultat de puret\'e pour le spectre
sur les formes diff\'erentielles. Dans cette introduction nous ne retenons que le corollaire suivant, conjecture $A^-$ de \cite{BC}. 

\begin{thm} \label{A-}
Pour $0 \leq k \leq n/2 -1$, il existe une constante strictement positive $\varepsilon (n, k)$ telle que pour
toute vari\'et\'e hyperbolique de congruence non trialitaire $M$ la premi\`ere valeur propre non nulle $\lambda_1^k = \lambda_1^k (M)$ du laplacien sur les $k$-formes diff\'erentielles $L^2$ v\'erifie~:
$$\lambda_1^k \geq \varepsilon (n,k) .$$
\end{thm}

\noindent
{\it Remarque.} Dans un article r\'ecent \cite{SpehVenky} B. Speh et T. N. Venkataramana ont 
d\'emontr\'e le r\'esultat d'isolation de la valeur propre nulle (th\'eor\`eme \ref{A-}) s'il est vrai en degr\'e
m\'edian (pour $n$ pair) pour tout $n$. Comme on le voit, ce dernier r\'esultat est vrai mais la 
d\'emonstration donne directement un r\'esultat (bien meilleur) pour les autres degr\'es.

\noindent

Ce th\'eor\`eme a un certain nombres de cons\'equences sur la topologie des vari\'et\'es hyperboliques
de congruence; cons\'equences que nous avons rassembl\'ees sous le nom de {\it propri\'et\'es de Lefschetz automorphes} dans \cite{IMRN,BC,SMF}. Nous supposerons -- pour simplifier -- que les
vari\'et\'es hyperboliques $\Gamma \backslash \H^n$ sont compactes et renvoyons aux
articles cit\'es ci-dessus pour les \'enonc\'es dans le cas g\'en\'eral. Dor\'enavant nous supposerons
donc $G$ anisotrope. 

\subsection{Propri\'et\'es de Lefschetz automorphes} Celles-ci pr\'edisent un lien entre entre la (co-)homologie des vari\'et\'es hyperboliques 
$\Gamma \backslash \H^n$ et les \og sections hyperplanes \fg \
donn\'ees par leur sous-vari\'et\'es compactes totalement g\'eod\'esiques. 
Une telle sous-vari\'et\'e est associ\'ee \`a un sous-groupe $H \subset G$ sur $\QQ$ stable par 
l'involution de Cartan de $G$. Le groupe $H(\RR)$ est alors localement isomorphe au produit
du groupe $\SO (k,1)$ (avec $k \leq n$) par un groupe compact. Au niveau des espaces sym\'etriques,
on obtient un plongement totalement g\'eod\'esique $\H^k \hookrightarrow \H^n$ et si $\Gamma \subset G$
est un sous-groupe de congruence l'inclusion ci-dessus passe au quotient en une immersion
totalement g\'eod\'esique 
$$j = j(\Gamma ) : (\Gamma \cap H) \backslash \H^k \rightarrow \Gamma \backslash \H^n .$$
Cette immersion induit une application naturelle entre groupes d'homologie~:
\begin{eqnarray} \label{jstar}
j_{\bullet} : H_{\bullet} ((\Gamma \cap H) \backslash \H^k ) \rightarrow H_{\bullet} (\Gamma \backslash \H^n) .
\end{eqnarray}
Si maintenant $g \in G(\QQ)$, on peut consid\'erer l'immersion 
$$j_g = j_g (\Gamma ) : (H \cap g^{-1} \Gamma g) \backslash \H^k \rightarrow \Gamma \backslash \H^n .$$
On d\'efinit alors l'{\it application de restriction virtuelle}
\begin{eqnarray} \label{resvirt}
H^{\bullet} (\Gamma \backslash \H^n ) \rightarrow \prod_{ g \in G (\QQ)} H^{\bullet} ((H \cap g^{-1} \Gamma g) 
\backslash \H^k )
\end{eqnarray}
induite par les applications de restriction duales aux immersions $(j_g )$. En passant \`a
la limite sur le syst\`eme inductif des sous-groupes de congruence $\Gamma \subset G$, on 
obtient finalement les application naturelles~:
\begin{eqnarray}
H_{\bullet} ({\rm Sh}^0 H ) \rightarrow H_{\bullet} ({\rm Sh}^0 G) \ \ \ \mbox{ et } \ \ \ H^{\bullet} ({\rm Sh}^0 G) \rightarrow \prod_{g \in G (\QQ )} H^{\bullet} ({\rm Sh}^0 H) ,
\end{eqnarray}
o\`u
$$H_{\bullet} ({\rm Sh}^0 G) = \lim_{\substack{\longleftarrow \\ \Gamma \ {\rm cong}}} H_{\bullet} (\Gamma \backslash \H^n ), \ \ \  H^{\bullet} ({\rm Sh}^0 G) = \lim_{\substack{\longrightarrow \\ \Gamma \ {\rm cong}}} H^{\bullet} (\Gamma \backslash \H^n )$$
et de m\^eme pour $H$. 

En ces termes, les deux th\'eor\`emes suivants d\'ecoulent de \cite[Theorem 2.4]{IMRN} et du th\'eor\`eme \ref{A-}, pour le premier, et de la d\'emonstration de \cite[Th\'eor\`eme 8.7]{SMF} (voir aussi 
la Conjecture 1.13), pour le second.  

\begin{thm} \label{th1}
Soient $H \subset G$ deux groupes semi-simples sur $\QQ$. Supposons $G(\RR ) \cong \SO(n,1)
\times ({\rm compact})$, $H(\RR ) \cong \SO (k,1) \times  ({\rm compact})$, 
$H$ invariant par une involution de Cartan de $G$ et tous deux non trialitaires. Alors,
\begin{enumerate}
\item pour tout entier $i \geq n/2$, l'application naturelle
$$H_i ({\rm Sh}^0 H ) \rightarrow H_i ({\rm Sh}^0 G)$$
est injective,
\item pour tout entier $i \leq k/2$, l'application naturelle
$$H^i ({\rm Sh}^0 G) \rightarrow \prod_{g \in G (\QQ )} H^i ({\rm Sh}^0 H)$$
est injective.
\end{enumerate}
\end{thm}

Noter qu'un r\'esultat proche du th\'eor\`eme \ref{th1} est d\'emontr\'e dans \cite{BHW}.

\begin{thm} 
Soit $G$ un groupe semi-simple sur $\QQ$. Supposons $G(\RR ) \cong \SO(n,1)
\times ({\rm compact})$ et $G$ non trialitaire. Soient $\alpha$ et $\beta$ deux classes de cohomologie de degr\'es respectifs $k$ et $l$ dans
$H^{\bullet} (Sh^0 G)$ avec ${k+l \leq  n/2}$. Il existe alors un \'el\'ement $g \in G({\Bbb Q})$
tel que
$$g(\alpha ) \wedge \beta \neq 0 $$
dans $H^{k+l} (Sh^0 G)$.
\end{thm}

Comme dans \cite{IMRN} on peut \'egalement consid\'erer le cas o\`u $G$ est un groupe 
unitaire tel que $G(\RR)$ soit isomorphe au produit de $\SU (n,1)$ par un groupe compact.
L'espace sym\'etrique associ\'e est l'espace hyperbolique complexe $\H_{\CC}^n$. On note
$H^{\bullet , 0} ({\rm Sh}^0 G)$ les groupes de cohomologie {\it holomorphe} obtenus en passant
\`a la limite sur les sous-groupes de congruence. Le th\'eor\`eme suivant d\'ecoule \'egalement
du th\'eor\`eme \ref{A-}, voir \cite{SMF}. 
Remarquons que l'inclusion $\SO (n,1) \subset \SU (n,1)$ correspond
au plongement totalement g\'eod\'esique -- et totalement r\'eel -- $\H^n \subset \H^n_{\CC}$.

\begin{thm} \label{RUO}
Soient $H \subset G$ deux groupes semi-simples sur $\QQ$. Supposons $H(\RR ) \cong \SO(n,1)
\times ({\rm compacte})$, $G(\RR ) \cong \SU (n,1) \times  ({\rm compacte})$ et $H$ invariant par une involution de Cartan de $G$. Alors, pour tout entier $i \leq n/2$, l'application naturelle
$$H^{i,0} ({\rm Sh}^0 G) \rightarrow \prod_{g \in G (\QQ )} H^i ({\rm Sh}^0 H)$$
est injective.
\end{thm}

Comme l'ont fait remarquer Raghunathan et Venkataramana dans \cite{RV} si $H$ est un groupe obtenu -- par restriction des scalaires -- \`a partir d'un groupe sp\'ecial orthogonal sur un corps de nombres totalement r\'eel qui n'est pas une forme exceptionnelle de $\SO (8)$ ou de $\SO(4)$, alors il existe un groupe unitaire $G$ sur $\QQ$, obtenu par restriction des scalaires \`a partir d'un
\og vrai \fg \ groupe unitaire -- c'est-\`a-dire associ\'e \`a une alg\`ebre de matrices -- et tel que $H \subset G$
soit invariant par une involution de Cartan de $G$. De plus si $H(\RR ) \cong \SO(n,1)
\times ({\rm compacte})$ on peut supposer $G(\RR ) \cong \SU (n,1) \times  ({\rm compacte})$. 
Or Anderson montre dans \cite{Anderson} que pour un tel groupe unitaire $G$ et pour tout
entier $i \leq n$, $H^{i,0} ({\rm Sh}^0 G) \neq \{ 0 \}$. Il d\'ecoule donc du th\'eor\`eme \ref{RUO}
le corollaire -- semble-t-il nouveau, voir \cite{Li} pour des r\'esultats partiels -- suivant.

\begin{cor} \label{Cor}
Soit $\Gamma$ un r\'eseau arithm\'etique dans le groupe de Lie r\'eel $\SO (n,1)$, $n \geq 2$.
Si $n=7$ nous supposons de plus que $\Gamma$ ne provient pas d'une forme tordue ${}^{6} \- D_4$.
Si $n=3$ et $\Gamma$ est commensurable au groupe des unit\'es d'une alg\`ebre de quaternions sur
un corps de nombres $L$ avec un unique plongement complexe, nous supposons de plus que
$L$ contient un sous-corps d'indice $2$. 

Alors, il existe un sous-groupe d'indice fini $\Gamma ' \subset \Gamma$ -- que l'on peut choisir de
congruence -- tel que pour tout $i \leq n$, $b_i (\Gamma ')$ -- le $i$-\`eme nombre de Betti
de la vari\'et\'e $\Gamma' \backslash \H^n$ -- est non nul.
\end{cor}

Une conjecture c\'el\`ebre (attribu\'ee \`a Thurston dans \cite{Borel}) affirme -- au moins pour le premier nombre de Betti -- qu'un r\'esultat similaire devrait \^etre vrai pour tout r\'eseau dans
$\SO (n,1)$. Dans \cite{GAFA} l'analogue du corollaire \ref{Cor} est v\'erifi\'e pour les r\'eseaux
non-arithm\'etiques construits par Gromov et Piatetski-Shapiro \cite{GPS}. Le cas
g\'en\'eral est encore ouvert. C'est \'egalement le cas pour la plupart des r\'eseaux arithm\'etiques
exclus dans le corollaire \ref{Cor}, \`a l'exception notable de ceux couverts par les r\'esultats de Clozel 
\cite{Clozelb1} et Rajan \cite{Rajan}. 

\subsection{Groupes trialitaires} Le cas des r\'eseaux arithm\'etiques exceptionels dans $\SO (7,1)$ provenant d'une forme tordue 
${}^{3,6} \- D_4$ est particuli\`erement int\'eressant. Commen\c{c}ons par remarquer que ces r\'eseaux
proviennent en fait n\'ecessairement d'une forme tordue ${}^6 \- D_4$. De tels r\'eseaux existent bel et bien --
cela r\'esulte par exemple de \cite{PrasadRapinchuk} -- et sont tous cocompacts. 

Dans cet article
on exclut le cas de ces formes exceptionnelles; les groupes consid\'er\'es sont alors des formes int\'erieures de groupes quasi-d\'eploy\'es auxquels la th\'eorie d'Arthur s'applique. Nous d\'etaillons 
le cas des formes tordues ${}^6 \- D_4$ dans un travail en pr\'eparation. 
Au prix d'un grand nombre d'efforts techniques le th\'eor\`eme \ref{A-} devrait s'\'etendre \`a ces formes
tordues. Nous ne le v\'erifions pas -- les corollaires mentionn\'es plus haut sont essentiellement vides. 
Nous montrons par contre le th\'eor\`eme suivant qui contraste avec le corollaire \ref{Cor} et vient confirmer la conjecture 4.6 de \cite{JIMJ}.

\begin{thm} \label{Tb0}
Soit $G$ un groupe un groupe projectif orthogonal ou de spin du type ${}^6 \- D_4$ d\'efini 
sur un corps de nombres totalement r\'eel. 
Alors, pour tout sous-groupe de congruence $\Gamma \subset G$,
$$b_1 (\Gamma ) =0.$$
\end{thm}

Le probl\`eme des sous-groupes de congruence reste ouvert pour ces groupes alg\'ebriques; il n'est
donc pas exclu que la conjecture de Thurston soit v\'erifi\'ee mais, contrairement \`a tous les autres cas
arithm\'etiques
connus, pour la d\'emontrer il sera n\'ecessaire de consid\'erer des sous-groupes qui ne sont pas de congruence.

On l'a dit, le th\'eor\`eme ci-dessus n'est pas vide, il existe de tels $G$.
Concluons par la construction -- due \`a Allison \cite[Example 11.8]{Allison} -- d'un exemple.

Soit $B = \QQ [b_0]$, o\`u le polyn\^ome minimal de $b_0$ est $x^4+3x-\frac12$. 
L'alg\`ebre $B$ est de dimension $4$ sur $\QQ$, elle est naturellement munie d'une trace $t$ et 
d'une involution $\theta : b \mapsto -b + \frac12 t(b)$. 
On peut associer \`a $B$ une alg\`ebre \`a involution -- non associative -- de dimension $8$~:
$${\rm CD} (B , -3) := B \oplus s_0 B,$$
muni du produit 
$$(b_1 + s_0 b_2 ) (b_3 + s_0 b_4 ) = b_1 b_3 + \mu (b_2 b_4^{\theta} )^{\theta} + s_0 (b_1^{\theta} b_4
+ (b_2^{\theta} b_3^{\theta})^{\theta} ) $$
et de l'involution 
$$\overline{b_1 + s_0 b_2} = b_1 - s_0 b_2^{\theta} .$$
Notons 
$$\gamma = \left( 
\begin{array}{ccc}
1 & & \\
& 1 & \\
& & 1 
\end{array} \right)$$
et 
$${\cal P} = \{ X \in M_3 ({\rm CD} (B , -3) ) \; : \; \gamma^{-1} {}^t \- \overline{X} \gamma = -X \mbox{ et } {\rm tr} (X) =0 \}.$$
L'alg\`ebre ${\cal K} := {\cal K} ({\rm CD} (B , -3))$ obtenue en formant la somme directe de l'alg\`ebre des d\'erivations int\'erieures de ${\rm CD} (B , -3)$ et de l'alg\`ebre ${\cal P}$ est une alg\`ebre de type
$D_4$. Sur $\CC$, l'alg\`ebre ${\rm CD} (B , -3)$ est isomorphe \`a l'alg\`ebre de Cayley 
${\cal C}_{\CC}$. Soient $n$ et $t$ respectivement la norme et la trace de ${\cal C}_{\CC}$ et
$$\langle x,y,z \rangle = \frac12 t (x(yz)). $$
La forme trilin\'eaire $\langle \cdot , \cdot , \cdot \rangle$ v\'erifie 
$$\langle x,y,z \rangle = \langle z,x,y \rangle = \langle \overline{y} , \overline{x} , \overline{z} \rangle \ \ \ 
\mbox{ pour } x, y, z \in {\cal C}_{\CC} .$$
On peut alors v\'erifier que l'alg\`ebre ${\cal K}_{\CC}$ est isomorphe \`a l'alg\`ebre 
$${\cal L} = \{ (L_1 , L_2 , L_3 ) \in \mathfrak{o} (n)^3 \; : \; \langle 
L_1 x , y , z \rangle + \langle x, L_2 y , z \rangle + \langle x,y,L_3 z \rangle = 0 \mbox{ pour } x, y, z \in {\cal C}_{\CC} \}.$$
On peut donc identifier ${\cal K}$ \`a une $\QQ$-forme de ${\cal L}$ et la sous-alg\`ebre associative
engendr\'ee par ${\cal K}$ dans ${\rm End}_{\CC} ({\cal C}_{\CC} )^3$ -- appel\'e {\it invariant d'Allen}
de ${\cal K}$ -- est l'alg\`ebre $M_4 (D)$ o\`u $D$ est l'alg\`ebre de quaternions 
$\left(\frac{\nu , -3}{K} \right)$ sur 
$K= \QQ (\nu )$ extension cubique -- de groupe de Galois $\mathfrak{S}_3$ -- associ\'ee 
\`a un \'el\'ement $\nu$ de polyn\^ome minimal $h(x)=x^3 +2x -9$.

Le polyn\^ome $h$ poss\`ede une racine r\'eelle et deux racines complexes conjugu\'ees de 
telle mani\`ere que $D_{\RR} \cong M_2 (\RR ) \oplus M_2 (\CC )$ et ${\rm sgn} ({\cal K}_{\RR} ) = -14$.
L'alg\`ebre ${\cal K}_{\RR}$ est donc isomorphe \`a $\mathfrak{o} (7,1)$ et les groupes -- projectif orthogonal ou de spin -- associ\'es
\`a l'alg\`ebre de Lie ${\cal K}$ sont des exemples de groupes auxquels le th\'eor\`eme \ref{Tb0} s'applique.

\section{Groupes orthogonaux}

On note $G$ un groupe sp\'ecial orthogonal sur un corps de nombres totalement r\'eel $F$. Dans
cette section on suppose $G$ non trialitaire, c'est-\`a-dire que l'on exclut le cas des formes exceptionnelles de $\SO (8)$. Soit ${\mathbf A}$ l'anneau des ad\`eles sur $F$.

Nous supposons $G$ compact \`a toutes les places \`a l'infini de $F$ sauf une -- not\'ee $v_0$ -- o\`u le
groupe est isomorphe au groupe $\SO(p,q)$. Supposons $p\geq q$ et notons $m=p+q$ le nombre de variables du groupe orthogonal, $\ell$ la partie enti\`ere de $m/2$ et $N=2\ell$. Le groupe $G$ est 
toujours forme int\'erieure d'un groupe quasi-d\'eploy\'e $G^*$. 

On d\'ecrit maintenant le groupe $G^*$ en distinguant deux cas selon la parit\'e de $m$.

\subsection{} Supposons d'abord $m=N+1$ impair. Alors, le groupe $G$ est forme int\'erieure du groupe orthogonal
{\it d\'eploy\'e} $G^* = \SO (m)$ sur $F$ associ\'e \`a la forme bilin\'eaire symm\'etrique attach\'ee \`a 
$$\left(
\begin{array}{ccc}
0 & & 1 \\
 & \adots & \\
1 &  & 0 
\end{array} \right).$$ 
Le groupe dual (complexe) de $G^*$ est $\widehat{G} = \Sp (N, \CC)$, le groupe symplectique de la forme altern\'ee sur $\CC^N$ de matrice 
$$\hat{J} = \left(
\begin{array}{cccccc}
& & & & & -1 \\
& & & & \adots & \\
& & & -1 & &  \\
& & 1 & & & \\
& \adots & & & & \\
1 & & & & &  
\end{array} \right),$$ 
et ${}^L \- G = \widehat{G} \times \Gamma_F$.

\subsection{} Supposons maintenant $m=N$ pair. On note $\SO (N)$ le groupe orthogonal
{\it d\'eploy\'e}  sur $F$ associ\'e \`a la forme bilin\'eaire symm\'etrique attach\'ee \`a 
$$J = \left(
\begin{array}{cc}
0 &  1_{\ell} \\
1_{\ell} &  0 
\end{array} \right).$$
Les formes {\it quasi-d\'eploy\'ees} de $\SO (N)$ sont param\'etr\'es par les morphismes de 
$\Gamma_F  = {\rm Gal} (\overline{{\mathbf Q}} / F)$ vers $\ZZ / 2\ZZ$. D'apr\`es la th\'eorie du corps de classes ces morphismes correspondent aux caract\`eres $\eta$ de $F^* \backslash {\mathbf A}^*$ tels que $\eta^2 = 1$ -- caract\`eres d'Artin d'ordre deux. Dans la suite nous notons $\SO (N , \eta)$ le groupe orthogonal tordu obtenu en faisant agir $\Gamma_F$ sur la diagramme de Dynkin -- ou la forme de $\SO (2)$ si $N=2$ -- {\it via} le caract\`ere $\eta$.
Lorsque $m=N$ est pair, il existe un caract\`ere d'Artin d'ordre deux $\eta$
tel que le groupe $G$ soit forme int\'erieure du groupe {\it quasi-d\'eploy\'e} 
$G^* = \SO(N,\eta)$. Le groupe dual (complexe) de $G^*$ est alors $\widehat{G} = \SO (N, \CC)$ et 
${}^L \- G = \widehat{G} \rtimes \Gamma_F$, o\`u $\Gamma_F$ op\`ere sur $\widehat{G}$ par un
automorphisme d'ordre $2$ -- trivial sur le noyau de $\eta$ -- et respectant un \'epinglage. (Pour une
description explicite, voir \cite[p. 79]{BC}.) Notons que le caract\`ere $\eta$ est trivial \`a l'infini 
si et seulement si $(p-q)/2$ est pair. Au niveau
des groupes r\'eels cela revient \`a la dichotomie~: si $(p-q)/2$ est impair, $\SO(p,q)$ est forme
int\'erieure de $\SO (\ell -1 ; \ell +1)$, si $(p-q)/2$ est pair, $\SO (p,q)$ est forme int\'erieure de $\SO (\ell ,
\ell)$ (d\'eploy\'e sur $\RR$). 

\subsection{} Nous supposons $G_{v_0}= \SO (p,q)$ d\'efini par la forme de matrice 
$$\left( 
\begin{array}{c|ccc}
1_{p-q} & & & \\
\hline 
& & & 1 \\
& & \adots & \\
& 1 & & 
\end{array} \right).$$
Une sous-alg\`ebre de Cartan $\mathfrak{t}$ de $\mathfrak{g}=\mathfrak{so} (p,q)$ est alors donn\'ee par les matrices 
\begin{eqnarray} 
X = \left(
\begin{array}{ccc|cccc}
& -x_1 & & &  & & \\ 
x_1 & & & & & & \\
& & \ddots &  & & & \\
\hline 
& & &  \ddots & & &  \\
& &  & & x_{\ell} & &  \\
& & & & & -x_{\ell} &  \\
& & & & &  & \ddots 
\end{array}
\right)  \ \ \ (x_i \in \RR ).
\end{eqnarray}

Posons $y_i = \sqrt{-1} x_i$, pour $i \leq \ell - q$, et $y_i = x_i$, pour $i > \ell -q$. Les coordonn\'ees $y$
donnent un isomorphisme
$$\mathfrak{t}_{\CC} = \mathfrak{t} \otimes \CC \cong \CC^{\ell}$$
pour lequel les racines dans $\Delta (\mathfrak{g}_{\CC} , \mathfrak{t}_{\CC})$ sont r\'eelles. Le groupe de Weyl 
correspondant $W$ est $\mathfrak{S}_{\ell} \ltimes \{ \pm 1\}^{\ell}$, dans le premier cas, alors que 
dans le deuxi\`eme cas $W = \mathfrak{S}_{\ell} \ltimes \{ \pm 1\}^{\ell -1}$, 
$\{ \pm 1 \}^{\ell -1}$ \'etant le sous-groupe de $\{ \pm 1 \}^{\ell}$, op\'erant diagonalement, 
d\'efini par $\small\prod s_i =1$. Notons $\mathfrak{t}_s$ l'espace vectoriel r\'eel engendr\'e par les $y_i$. Cet
espace s'identifie \`a une sous-alg\`ebre de Cartan d\'eploy\'ee de la forme r\'eelle d\'eploy\'ee de $G_{v_0}$; il s'identifie en particulier au sous-espace correspondant pour $G^*$.

\section{Spectre automorphe des formes quasi-d\'eploy\'ees} \label{qs}

Dans ce chapitre nous supposons $G$ quasi-d\'eploy\'e, autrement dit $G=G^*$. Dans ce cas les 
r\'esultats globaux d'Arthur \cite[\S 30]{Arthur} s'appliquent. En particulier le th\'eor\`eme 30.2 donne une d\'ecomposition de la partie 
discr\`ete du spectre automorphe de $G({\mathbf A})$ selon certains param\`etres globaux 
$\tilde{\psi} \in \widetilde{\Psi}_2 (G)$. En la place archim\'edienne $v_0$ un tel param\`etre 
$\tilde{\psi}$ se localise en un 
param\`etre local $\tilde{\psi}_{v_0}$ \'egal \`a une somme directe formelle
\begin{eqnarray} \label{param}
\tilde{\psi}_{v_0} = \psi_1 \boxplus \cdots \boxplus \psi_r ,
\end{eqnarray}
o\`u chaque $\psi_j$ est la localisation en $v_0$ d'un param\`etre global discret pour 
$\GL (N_j)_{|F}$ avec 
$$N= N_1 + \ldots + N_r.$$
Les rangs $N_j$ sont des entiers strictement positifs de la forme $N_j=m_j n_j$ avec $m_j , n_j \geq 1$
et chaque param\`etre $\psi_j$ est le produit tensoriel formel 
$\psi_j = \mu_j \boxtimes r_j$ d'une repr\'esentation irr\'eductible $\mu_j$ de $\GL (m_j , \RR)$ -- composante archim\'edienne d'une repr\'esentation automorphe cuspidale de $\GL (m_j)$ -- 
et de l'unique repr\'esentation irr\'eductible $r_j$ de 
dimension $n_j$ du groupe $\SL (2, \CC)$. Enfin chaque param\`etre $\psi_j$ est autodual, 
{\it i.e.} isomorphe \`a $\psi_j^{\theta} = \mu_j^{\theta} \boxtimes r_j$, o\`u
$\mu_j^{\theta}$ est la repr\'esentation contragr\'ediente de $\mu_j$. 

\subsection{Param\`etres d'Arthur (g\'en\'eralis\'es)} Notons $\varphi_j : W_{\RR} \rightarrow {}^L \- \GL (m_i)$ le param\`etre de Langlands associ\'e \`a la
repr\'esentation $\mu_j$. On pr\'ef\`erera voir $\tilde{\psi}_{v_0}$ 
comme la repr\'esentation de dimension $N$ du groupe $W_{\RR} \times \SL (2 , \CC)$ \'egale \`a
$\bigoplus_j \varphi_j \otimes r_j$. Notons $\psi : \CC^{\times} \times \SL (2 , \CC) \rightarrow \GL (N , \CC)$ sa restriction \`a $W_{\CC} = \CC^{\times}$. Chacun des param\`etres $\varphi_j$ se restreint \`a $\CC^{\times}$ en une 
repr\'esentation semi-simple donn\'ee par $m_j$ (quasi-)caract\`eres $\chi$ de la forme $z^p \overline{z} \- {}^q$ avec $p-q \in \ZZ$ ; 
puisque $\tilde{\psi}$ est un param\`etre r\'eel, 
si $\chi$ appara\^{\i}t dans $\varphi_j$ alors $\chi^{\sigma}$ appara\^{\i}t aussi dans $\psi$ (mais
peut-\^etre pour un indice diff\'erent de $j$). Ici on a not\'e
$\chi^{\sigma} (z) =\chi (\overline{z})$. 
 
La conjecture de Ramanujan g\'en\'eralis\'ee impliquerait que chacun de ces caract\`eres
est unitaire, {\it i.e.} ${\rm Re} (p+q) =0$. \`A d\'efaut, le th\'eor\`eme 
de Luo, Rudnick et Sarnak \cite{LuoRudnickSarnak} -- tel qu'\'etendu dans \cite[Chapitre 7]{BC} --
implique que $| {\rm Re} (p + q ) | < 1 - \frac{2}{m_j^2 +1}$. La condition d'auto-dualit\'e force 
chaque caract\`ere $\chi$ \`a appara\^{\i}tre avec son dual $\chi^{-1}$. Enfin 
$\psi $ se factorise \`a travers $\widehat{G}$ puisque $\tilde{\psi} \in \widetilde{\Psi}_2 (G)$.  

Appelons donc {\it param\`etre d'Arthur g\'en\'eralis\'e} toute repr\'esentation 
$$\psi = \bigoplus_{j=1}^r \varphi_j \otimes r_j : \CC^{\times} \times \SL (2 , \CC ) \rightarrow \widehat{G},$$
o\`u chaque $\varphi_j$ est une repr\'esentation semi-simple de $\CC^{\times}$ de rang $m_j$, $r_j$
est la repr\'esentation irr\'eductible de 
dimension $n_j$ du groupe $\SL (2, \CC)$, $N= \sum_j n_j m_j$ et si $\chi = z^p \overline{z} \- {}^q$ est
un caract\`ere apparaissant dans un $\varphi_j$, $\chi^{\sigma}$ et $\chi^{-1}$ apparaissent aussi et~:
$$p-q \in \ZZ \ \ \ \mbox{ et } \ \ \ |{\rm Re} (p + q ) | < 1 -  \frac{2}{m_j^2 +1} .$$

\subsection{Caract\`ere infinit\'esimal} \`A tout param\`etre d'Arthur g\'en\'eralis\'e $\psi$ on associe un param\`etre
$$\begin{array}{cccl}
\varphi_{\psi} : & \CC^{\times} & \rightarrow & \widehat{G} \subset \GL (N, \CC) \\
& z & \mapsto & \psi \left( z , \left(
\begin{array}{cc}
(z \overline{z})^{1/2} & 0 \\
0 & (z \overline{z})^{-1/2}
\end{array} \right) \right) .
\end{array}$$
\'Etant semi-simple, il est (\`a conjugaison pr\`es) d'image contenu dans le tore maximal
$$\widehat{T} = \{ {\rm diag} (x_1 , \ldots , x_{\ell} , x_{\ell}^{-1} , \ldots , x_1^{-1} ) \} $$ de 
$\widehat{G}$. On peut donc \'ecrire $$\varphi_{\psi} = ( \eta_1 , \ldots , \eta_{\ell} , \eta_{\ell}^{-1} , 
\ldots , \eta_1^{-1} )$$ o\`u les $\eta_i$ sont des caract\`eres, de la formes 
$z^{P_i} \overline{z} \- {}^{Q_i}$. On v\'erifie ais\'ement que le vecteur
$$\nu_{\psi} = (P_1 , \ldots , P_{\ell} ) \in \CC^{\ell} \cong \mathfrak{t}_{\CC}$$
est uniquement d\'efini modulo $W$; on l'appelle le {\it caract\`ere infinit\'esimal associ\'e \`a 
$\psi$}.  Noter que le param\`etre -- vu comme vecteur dans $\CC^{N} = \CC^{2 \ell}$ -- associ\'e \`a $\varphi_{\psi} : \CC^{\times} 
\rightarrow \GL (N, \CC)$ est $(\nu_{\psi } , -\nu_{\psi})$. La proposition suivante 
d\'ecoule du th\'eor\`eme 30.2 de \cite{Arthur} 
et de \cite[Lemmes 6.3.1 \& 6.4.1]{BC}.

\begin{prop} Soit $\pi$ est une repr\'esentation automorphe de $G$.
Alors, il existe un param\`etre d'Arthur g\'en\'eralis\'e $\psi$ tel que le caract\`ere infinit\'esimal de $\pi_{v_0}$
soit associ\'e \`a $\nu_{\psi}$.
\end{prop}

\section{Stabilisation de la formule des traces} \label{1.1}

On voudrait comparer les spectres automorphes discrets des groupes $G$ et $G^*$.

Soit donc $f = \otimes_v f_v$ une fonction lisse \`a support compact, d\'ecomposable, sur $G({\mathbf A})$. On s'int\'eresse essentiellement \`a 
\begin{eqnarray} \label{61}
{\rm trace}  \ R_{\rm dis}^G (f) = {\rm trace} \left( f | L^2_{\rm dis} ( G(F) \backslash G({\mathbf A})) \right),
\end{eqnarray}
o\`u $L_{\rm dis}^2$ est la partie discr\`ete de l'espace des formes automorphes sur $G(F) \backslash G({\mathbf A})$. 
Cette trace est bien d\'efinie d'apr\`es M\"uller \cite{Muller}, mais ceci n'est pas n\'ecessaire
pour les d\'emonstrations qui suivent. En effet, ce n'est pas l'expression (\ref{61}) que l'on peut comparer
\`a son analogue pour $G^*$, mais \og la partie discr\`ete de la formule des traces pour $G$ \fg. Il s'agit 
alors de fixer un r\'eel strictement positif $t$ et, selon Arthur, de consid\'erer des expressions relatives \`a $G$ et $G^*$ portant sur des repr\'esentations dont la norme du caract\`ere infinit\'esimal est \'egale \`a 
$t$. 

Pour $t$ fix\'e, Arthur d\'efinit une distribution $f \mapsto I_{{\rm disc} , t}^G (f)$, somme de la partie de
(\ref{61}) relative \`a $t$ et de termes associ\'es \`a diverses repr\'esentations induites de sous-groupes de Levi \cite[(21.19)]{Arthur}. On reviendra plus tard sur les termes compl\'ementaires. Observons 
que l'on ne sait pas {\it a priori} montrer que la somme sur $t$ de ces distributions
converge; cela n\'ecessiterait d'\'etendre le r\'esultat de M\"uller mentionn\'e plus haut. 

\subsection{Sous-groupes elliptiques} 
Consid\'erons la famille ${\cal E}_{\rm ell} (G)$ des donn\'ees endoscopiques elliptiques pour $G$
\cite[\S 27]{Arthur}. 
Puisque $G$ est forme int\'erieure de $G^*$, ${\cal E}_{\rm ell} (G) = {\cal E}_{\rm ell} (G^*)$. Cet ensemble est d\'ecrit explicitement par Arthur \cite[\S 30]{Arthur}. 
Il faut distinguer deux cas selon la parit\'e de $m$. 

Si $m=N+1$ est impair $G^* = \SO (N+1)$ et l'ensemble $ {\cal E}_{\rm ell} (G^*)$ est param\'etr\'e par des pairs d'entiers pairs $(N' , N'')$ avec $N'' \geq N' \geq 0$ et $N = N'+N''$.
Le groupe endoscopique correspondant est le groupe d\'eploy\'e
\begin{eqnarray} \label{62a}
H = \SO (N' +1) \times \SO (N'' +1).
\end{eqnarray}
Ainsi 
$$\widehat{H} = \Sp \left( \frac{N'}{2} \right) \times \Sp \left( \frac{N''}{2} \right)$$
et ${}^L \- H= \widehat{H} \times \Gamma_F$.

Si $m=N$ est pair $G^*=\SO (N , \eta )$ et l'ensemble $ {\cal E}_{\rm ell} (G^*)$ est param\'etr\'e par des pairs d'entiers pairs $N'' \geq N' \geq 0$ tels que $N = N'+N''$, et des paires correspondantes de
de caract\`eres d'Artin $\eta'$, $\eta ''$ avec $\eta ' \eta '' = \eta$. \footnote{Si $N' = 0$, $\eta ' =1$; si 
$N'$ ou $N''$ est \'egal \`a $2$, le caract\`ere correspondant est non trivial.} Le groupe endoscopique
correspondant est le groupe quasi-d\'eploy\'e 
\begin{eqnarray} \label{62b}
H = \SO (N' , \eta ') \times \SO (N'' , \eta '' ).
\end{eqnarray}
Ainsi 
$$\widehat{H} = \SO (N' , \CC) \times \SO (N'' , \CC)$$
et ${}^L \- H= \widehat{H} \rtimes \Gamma_F$, o\`u $\Gamma_F$ op\`ere sur chaque facteur par un automorphisme d'ordre $2$, respectant un \'epinglage.

Dans tous les cas, on v\'erifie l'existence d'un morphisme naturel ${}^L \- H \rightarrow {}^L \- G$. 

\subsection{Lemmes fondamentaux et transfert} 
Les hypoth\`eses d'analyse harmonique locale de \cite[\S 5]{ArthurSTF1} -- lemmes fondamentaux standard et pond\'er\'e -- sont maintenant des th\'eor\`emes, voir Ng\^o \cite{Ngo} et Chaudouard-Laumon \cite{CL}. Elles permettent d'associer \`a $f$ une famille de fonctions $(f^H)_{H \in 
{\cal E}_{\rm ell} (G)}$ (c'est une correspondance, non une application~: 
$f^H$ n'est d\'efinie que par ses int\'egrales
orbitales stables). Si $f$ est non ramifi\'ee hors d'un ensemble fini de places $S \supset \infty$, 
il en est de m\^eme de $f^H$ (si $H$ est ramifi\'e en une place $v$ et 
$f_v$ non ramifi\'ee, alors $f_v^H$ et donc $f^H$ est nulle). 

Le groupe $G^*$ appartient \`a ${\cal E}_{\rm ell} (G)$. C'est le seul de dimension maximale. On notera $f^*$ la fonction $f^{G^*}$. 

\begin{thm}[Arthur] 
On a~:
\begin{eqnarray} \label{63}
I_{{\rm disc} , t}^G (f) = \sum_{H \in {\cal E}_{\rm ell} (G)} \iota (G,H) S_{{\rm disc} , t}^H (f^H)
\end{eqnarray}
o\`u, pour tout $H$, $S_{{\rm disc} , t}^H$ est une distribution stable.
\end{thm}

Voir \cite[Cor. 29.10]{Arthur}. Les coefficients $\iota (G,H)$ sont d\'efinis dans \cite[\S 27]{Arthur}
et sont des rationnels $>0$. Noter que (\ref{63}) -- appliqu\'e \`a $G^*$ plut\^ot qu'\`a $G$, et 
inductivement \`a ses sous-groupes endoscopiques -- d\'efinit, de fa\c{c}on unique, les distributions
$S_{{\rm disc} , t}^H$. 

\section{D\'estabilisation de la formule des traces} \label{1.2}

Les termes de droite de (\ref{63}), d\'efinis \`a partir du c\^ot\'e g\'eom\'etrique de la formule des traces, n'ont pas {\it a priori} d'interpr\'etation spectrale. Pour utiliser cette identit\'e, il faut donc d\'estabiliser 
son membre de droite. 

\subsection{} Supposons pour un instant que $G$ est un groupe (r\'eductif, connexe) arbitraire sur un corps de nombres. Si $G$ est {\it quasi-d\'eploy\'e}, $G$ appara\^{\i}t dans (\ref{63}) comme l'un de ses sous-groupes 
endoscopiques et l'on a l'\'egalit\'e 
\begin{eqnarray} \label{64}
I_{{\rm disc} , t}^G (f) = S_{{\rm disc} , t}^G (f) + \sum_{H \in {\cal E}^0_{\rm ell} (G)} \iota (G,H) S_{{\rm disc} , t}^H (f)
\end{eqnarray}
o\`u la somme porte sur les groupes endoscopiques propres. Cf. \cite[(29.21)]{Arthur}, ainsi que la discussion suivant le Cor. 29.10; {\it ibid}. Nous pouvons appliquer (\ref{64}) \`a chacun des groupes endoscopiques $H$ apparaissant dans (\ref{63}). 

\subsection{Transfert} Fixons un ensemble fini $S$ de places de $F$, contenant les places archim\'ediennes, et supposons que $S$ contient les places de ramification de $G$. Pour $v \notin S$, $G \times F_v$ est isomorphe
au groupe quasi-d\'eploy\'e $G^* (F_v )$ et se d\'eploie sur un extension non-ramif\'ee de $F_v$; ce groupe poss\`ede donc un sous-groupe hypersp\'ecial (voir \cite[1.10.2]{Tits}) que l'on 
note $K_v$. Soit ${\cal H}_v$ l'alg\`ebre de Hecke correspondante.

On supposera que $f = f_S \otimes f^S$, o\`u $f^S = \otimes_{v \notin S} f_v$ et $f_v \in {\cal H}_v$
pour tout $v \notin S$. 

Si $v \notin S$, la correspondance $f_v \leadsto f_v^H$ est une application, d\'ecrite explicitement,
de ${\cal H}_v$ vers l'alg\`ebre de Hecke de $(H(F_v) , K_v^H)$ o\`u $K_v^H$ est un sous-groupe
hypersp\'ecial de $H(F_v)$, qui est non-ramifi\'e si $f^H \neq 0$. Les groupes $H$ ont \'et\'e d\'ecrits
dans le paragraphe pr\'ec\'edents, et d\'ependent, puisque $\eta '' = \eta ' \eta$, d'un caract\`ere d'Artin
quadratique $\eta '$. Ce caract\`ere \'etant non ramifi\'e hors $S$, ceci ne laisse qu'un nombre fini de
possibilit\'es pour $\eta '$ et donc pour les groupes $H$.

En la place r\'eelle $v_0$, on dispose des r\'esultats de Shelstad \cite{Shelstad} sur le 
transfert $f_{v_0} \mapsto f_{v_0}^H$ pour les fonctions dans les espace de Schwartz-Harish-Chandra.
Il d\'ecoule de \cite[Appendice, Th\'eor\`eme A.3]{ClozelDelorme} que les propri\'et\'es de support 
compact et de $K$-finitude peuvent aussi \^etre pr\'eserv\'ees. L'application de transfert n'a pas d'importance pour nous, seuls comptent son existence et le fait, que nous expliquons dans le paragraphe suivant, que {\it le transfert est compatible aux multiplicateurs d'Arthur}.

\subsection{Multiplicateurs d'Arthur} Rappelons que le centre ${\cal Z}$ de l'alg\`ebre enveloppante de $\mathfrak{g}$ s'identifie
\`a $S(\mathfrak{t}_{\CC})^W$. Toute repr\'esentation irr\'eductible admissible $\pi$
de $G_{v_0}$ a un caract\`ere infinit\'esimal que l'on voit comme un \'el\'ement 
$\nu_{\pi} \in \mathfrak{t}_{\CC}^* /W$. L'alg\`ebre ${\cal Z}$ agit sur ${\cal H}_{v_0}$ -- l'alg\`ebre des fonctions $K$-finies et \`a support compact sur $G_{v_0}$. Arthur montre plus g\'en\'eralement -- voir \cite[Theorem 4.2]{A9} ou encore 
\cite[Theorem 20.4]{Arthur} -- qu'il existe une action canonique de l'alg\`ebre 
${\cal E} (\mathfrak{t}_s )^W$ sur ${\cal H}_{v_0}$. Ici ${\cal E} (\mathfrak{t}_s )^W$ d\'esigne 
l'alg\`ebre des {\it multiplicateurs} -- distributions $W$-invariantes et \`a support compact sur 
$\mathfrak{t}_s$ munis du produit de convolution. Arthur montre en outre que pour toute 
repr\'esentation irr\'eductible admissible $\pi$ de $G_{v_0}$,
\begin{eqnarray} \label{Amult}
\pi (\alpha \cdot f) = \widehat{\alpha} (\nu_{\pi}) \pi (f) , \ \ \ \alpha \in 
{\cal E} (\mathfrak{t}_s )^W , \ f \in {\cal H}_{v_0} ,
\end{eqnarray}
o\`u $\alpha \mapsto \widehat{\alpha}$ est la transform\'ee de Fourier qui est en particulier une fonction
holomorphe $W$-invariante sur $\mathfrak{t}_{\CC}^*$. 

\subsection{} Revenons maintenant au transfert. Soit $\varphi_H : W_{\RR} \rightarrow {}^L \- H$ un param\`etre
de Langlands r\'eel pour le groupe endoscopique $H$ et $\varphi : W_{\RR} \rightarrow {}^L \- G$
le param\`etre de Langlands r\'eel obtenu en composant $\varphi_H$ par le morphisme naturel
${}^L \- H \rightarrow {}^L \- G$.  Notons $\Pi_H$ et $\Pi$ les $L$-paquets correspondants 
de repr\'esentations admissibles de $H_{v_0}$ et de $G_{v_0}$. 

On peut identifier le tore maximal $\widehat{T}_H$ de $\widehat{H}$ au tore
maximal $$\widehat{T} = \{ {\rm diag} (x_1 , \ldots , x_{\ell} , x_{\ell}^{-1} , \ldots , x_1^{-1}) \}$$ de 
$\widehat{G}$ ($=\Sp (2\ell , {\Bbb C})$ ou $\SO (2\ell , \CC)$). 
On en d\'eduit un isomorphisme entre les
tores maximaux $T_H$ et $T$ de $H$ et $G^*$. Par ailleurs $G$ et $G^*$ ont la m\^eme 
complexification, et leurs sous-alg\`ebres de Cartan complexifi\'ees s'identifient donc 
canoniquement (modulo l'action du groupe de Weyl $W=W(G^* , T)$). Cet isomorphisme permet de 
r\'ealiser le groupe de Weyl $W_H := W (H, T_H)$ comme sous-groupe du groupe de Weyl $W$ 
et induit un isomorphisme 
$$\mathfrak{t}_{H, \CC}^* / W_H \rightarrow \mathfrak{t}_{\CC}^* /W.$$ 
Les param\`etres 
$(\varphi_H) _{\CC^{*}} : \CC^{\times} \rightarrow \widehat{H}$ et $(\varphi)_{\CC^{\times}} : \CC^{\times} \rightarrow \widehat{G}$
\'etant semi-simples, leurs images sont alors (\`a conjugaison pr\`es) contenues dans $\widehat{T}$. 
Il r\'esulte de \cite[Lemmes 6.3.1 \& 6.4.1]{BC} que les param\`etres diagonaux associ\'es \`a 
$\varphi_H : \CC^{\times} \rightarrow \GL(N , \CC)$ et $\varphi : \CC^{\times} \rightarrow \GL (N , \CC)$ sont respectivement 
$(\nu_{\Pi_H} , -\nu_{\Pi_H})$ et $(\nu_{\Pi} , - \nu_{\Pi})$ (modulo $\mathfrak{S}_N$), o\`u 
$\nu_{\Pi_H}$ et $\nu_{\Pi} \in \CC^{\ell}$ sont les 
caract\`eres infinit\'esimaux respectifs des membres des $L$-paquets $\Pi_H$ et $\Pi$. Nous
dirons que les caract\`eres infinit\'esimaux $\nu_{\Pi}$ et $\nu_{\Pi_H}$ sont 
{\it associ\'es par fonctorialit\'e}. De la m\^eme mani\`ere on peut parler de 
multiplicateurs $\alpha_H \in {\cal E} (\mathfrak{t}_{H,s} )^{W_H}$ et $\alpha \in {\cal E} (\mathfrak{t}_s )^W$ associ\'es par fonctorialit\'e; cela revient \`a demander que
$$\widehat{\alpha} (\nu) = \widehat{\alpha}_H (\nu_H)$$
pour tout couple $(\nu , \nu_H)$ de caract\`eres infinit\'esimaux associ\'es par fonctorialit\'e.

Dire que le transfert est compatible aux multiplicateurs d'Arthur revient alors \`a dire que si $\alpha$
et $\alpha_H$ sont deux multiplicateurs associ\'es par fonctorialit\'e alors
\begin{eqnarray} \label{multComp}
{\rm trace} \  \Pi_H (\alpha_H \cdot f^H ) = \widehat{\alpha} (\nu_{\Pi}) \ {\rm trace}  \ \Pi_H (f^H ),
\end{eqnarray}
pour tout couple de $L$-paquets $(\Pi , \Pi_H)$ se correspondant par la fonctorialit\'e ${}^L H \rightarrow {}^L \- G$.

\subsection{D\'estabilisation}
Appliquons maintenant (\ref{64}) \`a chacun des termes de (\ref{63}). Par r\'ecurrence, on voit que 
$S_{{\rm disc}, t}^H (f^H)$ s'\'ecrit comme combinaison lin\'eaire des termes $I_{{\rm disc} , t}^J (f^J)$ 
o\`u les groupes $J$ sont des groupes endoscopiques it\'er\'es des groupes $H$.

Soit $H$ un groupe endoscopique associ\'e \`a une partition $N=N'+N''$, voir (\ref{62a}) ou (\ref{62b}).
Un groupe endoscopique pour $H$ se d\'ecompose en produit de groupes endoscopiques pour chacun
des facteurs et ceux-ci sont d\'ecrits au paragraphe pr\'ec\'edent. Si $H'$ est par exemple le 
premier facteur de (\ref{62a}), resp. (\ref{62b}), tout sous-groupe endoscopique elliptique $H_1$ pour
$H'$ est de la forme
$$H_1 = \SO (N_1 +1) \times \SO (N_2+1), \ \ \ {\rm resp.} \ H_1 = \SO (N_1 , \eta_1 ) \times \SO (N_2 , \eta_2),$$
o\`u $N_1$, $N_2$ sont pairs positifs et de somme $N'$ et $\eta_1$, $\eta_2$ sont des caract\`eres 
d'Artin d'ordre $2$ avec $\eta_1 \eta_2 = \eta '$. Comme dans le paragraphe pr\'ec\'edent, on a un 
morphisme naturel ${}^L \- H_1 \rightarrow {}^L \- H'$.

Apr\`es it\'eration (finissant par des groupes $\SO(3)$, resp. $\SO(2)$, qui n'ont pas de sous-groupes endoscopiques propres), on voit que les groupes endoscopiques it\'er\'es sont, dans le premier cas, 
des produits de groupes $\SO (N_i +1)$ et, dans le second cas, des produits de groupes $\SO(N_i ,
\eta_i)$ o\`u les $N_i$ sont des entiers pairs qui forment une partition de $N$ et les $\eta_i$ sont
des caract\`eres d'Artin d'ordre $2$ qui v\'erifient $\eta_1 \cdots \eta_r =1$. (Ce ne sont des groupes
endoscopiques pour $G$ que s'il y a deux facteurs.) Dans tous les cas, l'argument de ramification 
pr\'ec\'edent reste valide~: ils sont en nombre fini. Un groupe endoscopique it\'er\'e $J$ est d\'efini
par une it\'eration~:
\begin{eqnarray} \label{67}
(G, H_1 , \ldots , H_r=J) = {\cal J}
\end{eqnarray}
et l'application $f\mapsto f^{H_1} \mapsto (f^{H_1})^{H_2} \mapsto \cdots \mapsto f^J$ --
en fait, une {\it correspondance} entre fonctions lisses -- n'est pas \'evidemment ind\'ependante 
de la suite $(G , \ldots , J)$. On conviendra donc qu'un groupe endoscopique it\'er\'e est, non 
le groupe $J$, mais la suite ${\cal J}$ de (\ref{67}), et on note $f^{\cal J}$ le terme final.

Au terme de cette proc\'edure, on obtient dans tous les cas une expression finie~:
\begin{eqnarray} \label{68}
I_{{\rm disc} , t}^G (f) = I_{{\rm disc} , t}^{G^*} (f^* ) + \sum_{{\cal J}} \eta (G , {\cal J}) I_{{\rm disc} , t}^J (f^{\cal J})
\end{eqnarray}
o\`u les $\eta (G,{\cal J})$ sont des rationnels, peut-\^etre n\'egatifs ou nuls.

\subsection{} Supposons pour l'instant $G$ (r\'eductif connexe sur $F$) arbitraire. Il est temps de d\'ecrire le terme
$I_{{\rm disc} , t}^G (f)$. Il est d\'efini par Arthur dans \cite[\S 4]{ITF2} -- voir la formule (4.3)~-- ainsi
que dans \cite[(21.19)]{Arthur}~:
\begin{eqnarray} \label{69}
\nonumber I_{{\rm disc} , t}^G (f) & = & {\rm trace} \left( f | L^2_{{\rm disc} , t} (G(F) \backslash G( {\mathbf A}) ) \right) \\
& & + \sum_M \frac{|W_0^M |}{|W_0^G |} \sum_{s \in W(M)_{\rm reg}} |\det (s-1)|^{-1} {\rm trace} \left(
M_P (s,0) I_{P,t} (0,f) \right).
\end{eqnarray}
Le terme correctif porte sur les sous-groupes de Levi standard, propres, de $G$; l'op\'erateur $I_{P,t} (0,f)$
est d\'efini par l'action de $f$ dans une repr\'esentation de $G({\mathbf A})$ unitairement induite \`a partir d'une
somme finie (pour $f_{\infty}$ $K_{\infty}$-finie) de repr\'esentations du spectre discret de $M({\mathbf A})$.
Les d\'efinitions des autres termes n'ont pas d'importance pour nous, cf. \cite{ITF2,Arthur}. Qu'il nous 
suffise de dire que $M_P (s,0)$ est un op\'erateur d'entrelacement de la repr\'esentation induite. Nous 
appliquerons (\ref{69}) \`a chaque terme de (\ref{68}).

\section{Caract\`eres infinit\'esimaux des repr\'esentations automorphes de $G$}

Compte tenu de la convergence absolue des expressions dans (\ref{68}), 
la compatibilit\'e du transfert archim\'edien $K$-fini aux multiplicateurs d'Arthur (voir 
\cite[\S 2.15]{AC} ou encore \cite[Prop. 3.5.4]{Labesse}) permet de raffiner l'identit\'e (\ref{68}) en 
s\'eparant les caract\`eres infinit\'esimaux. Ceci fournit une identit\'e entre traces de repr\'esentations dont les caract\`eres infinit\'esimaux sont fix\'es et associ\'es par fonctorialit\'e; 
les sommes apparaissant dans (\ref{68}) sont alors finies. On a ainsi ramen\'e l'\'etude des caract\`eres
infinit\'esimaux de repr\'esentations automorphes de $G$ \`a l'analogue pour ses sous-groupes endoscopiques. Ceux-ci 
sont des produits de groupes quasi-d\'eploy\'es et les r\'esultats d'Arthur \cite[\S 30]{Arthur} que nous avons
rappel\'es au chapitre \ref{qs} s'appliquent. 

Fixons un groupe endoscopique it\'er\'e ${\cal J}$. Puisque $J$ est un produit de groupes quasi-d\'eploy\'es, 
il d\'ecoule du chapitre \ref{qs} que si $\pi$ est une repr\'esentation automorphe de $J$, il existe un param\`etre d'Arthur g\'en\'eralis\'e (ainsi d\'efini dans l'introduction du chapitre 2) $\psi$ tel que le caract\`ere infinit\'esimal de $\pi_{v_0}$ est \'egal \`a $\nu_{\psi}$.

Le paragraphe pr\'ec\'edent et la s\'eparation des caract\`eres infinit\'esimaux impliquent donc~:

\begin{thm} \label{C}
Si une repr\'esentation irr\'eductible $\pi$ de $G_{v_0}$ appara\^{\i}t (faiblement) dans $L^2 (\Gamma 
\backslash G)$ pour un sous-groupe de congruence $\Gamma$, 
il existe un param\`etre d'Arthur g\'en\'eralis\'e $\psi$ tel que le
caract\`ere infinit\'esimal de $\pi$ soit \'egal \`a $\nu_{\psi}$.
\end{thm}

On peut tirer de ce th\'eor\`eme des cons\'equences similaires \`a celles que l'on tire de la 
conjecture 6.1.2 dans \cite[\S 6.3 et 6.4]{BC}. Supposons donc $p=n$ et $q=1$. Notons 
$M = {\rm S}({\rm O}(n-1) \times {\rm O}(1,1)) \subset G_{v_0}$ le sous-groupe qui -- dans notre r\'ealisation de $G_{v_0} = \SO (n,1)$ -- stabilise la d\'ecomposition $\RR^{n+1} = \RR^{n-1} \oplus \RR^2$. 
Sa composante neutre ${}^0 \- M \cong \SO (n-1)$ est compacte. D'apr\`es la classification de 
Langlands, une repr\'esentation admissible irr\'eductible $\pi$ de $G_{v_0}$ est soit un membre de 
la s\'erie discr\`ete de $G$ -- auquel cas $n$ est pair -- soit un sous-quotient irr\'eductible $J(\tau , s)$
d'une induite (non n\'ecessairement unitaire) $I(\tau , s)$ o\`u $\tau \in \widehat{{}^0 \- M}$ et ${\rm Re} (s) \geq 0$ (voir \cite[\S 6.3 \& 6.4]{BC} pour plus de d\'etails).

\begin{prop} \label{PA}
Toute repr\'esentation unitaire $\pi$ de $G_{v_0}$ dont le caract\`ere infinit\'esimal est associ\'e 
\`a un param\`etre d'Arthur g\'en\'eralis\'e appartient \`a la s\'erie discr\`ete de $G_{v_0}$, ou est de la forme $J(\tau ,s)$
avec 
\begin{enumerate}
\item $|{\rm Re} (s) | < \frac12 - \frac{1}{N^2 +1}$, ou 
\item $s\in \frac{n-1}{2} + \ZZ$.
\end{enumerate}
\end{prop}
\begin{proof}[D\'emonstration] Le caract\`ere infinit\'esimal de la repr\'esentation $\pi= J(\tau ,s)$ est
$$\nu_{\pi} = (\lambda_{\tau} , s ) \in \CC^{\ell}$$
o\`u $\lambda_{\tau}$ -- le caract\`ere infinit\'esimal de $\tau$ -- est une suite strictement croissante 
d'\'el\'ements dans $(n-1)/2 + \ZZ$. Supposons $\nu_{\pi}$ associ\'e \`a un param\`etre d'Arthur 
g\'en\'eralis\'e $\psi : \CC^{\times} \times \SL (2 , \CC) \rightarrow \widehat{G}$. Montrons d'abord que si 
$s \notin \frac12 \ZZ$ alors $|{\rm Re} (s) | < \frac12 - \frac{1}{N^2 +1}$.

Supposons en effet $s \notin \frac12 \ZZ$ et \'ecrivons 
$\psi = \oplus_j \varphi_j \otimes r_j$
comme dans le chapitre \ref{qs}. Puisque $\nu_{\pi} = \nu_{\psi}$ il existe un indice $j$, un entier $k$
compris entre $1$ et $n_j$ (la dimension de $r_j$)  et un caract\`ere $\chi = z^p \overline{z} \- {}^q$
apparaissant dans $\varphi_j$ tels que 
$$s= p + \frac{n_j+1-2k}{2}.$$
Donc $p \notin \frac12 \ZZ$. Par ailleurs la repr\'esentation $\chi \otimes r_j$ de $\CC^{\times} \times \SL (2 , \CC)$
appara\^{\i}t avec sa duale $\chi^{-1} \otimes r_j$ et puisque $p\neq 0$, $\chi \neq \chi^{-1}$. Si 
$n_j >1$ le caract\`ere infinit\'esimal $\nu_{\psi}$ contient au moins deux coordonn\'ees $\notin \frac12 \ZZ$, contrairement \`a $\nu_{\pi}$. Donc $n_j =1$ et $s=p$. 

Remarquons maintenant que $\chi$ et $\chi^{-1}$ apparaissent tous deux dans 
$\varphi_j$ avec $|{\rm Re} (p+q) | < 1 - 2/(m_j^2 +1) \leq 1 - 2/ (N^2+1)$. Les caract\`eres $\chi$ et $\chi^{-1}$ sont les seuls ayant un exposant $\notin \frac12 \ZZ$. Ceci force $\chi^{\sigma} = \chi$ ou $\chi^{-1}$ soit 
$\chi = (z \overline{z})^p$ ou $(z/\overline{z} )^p$. Le second cas 
n'est pas possible puisque $p=s \notin \frac12 \ZZ$.
Seule reste comme possibilit\'e $\chi = (z \overline{z})^s$ avec $|{\rm Re} (s) | < \frac12 - \frac{1}{N^2 +1}$.

Nous pouvons maintenant supposer que $s \in \frac{n}{2} + \ZZ$. L'argument pr\'ec\'edent montre toujours
que $n_j = 1$ et $s=p$ comme au-dessus. Supposons $|{\rm Re} ( s) | > 1 - 2/ (N^2+1)$. Alors l'argument pr\'ec\'edent montre de plus que $\chi = 
(z/\overline{z} )^p$. Vue comme repr\'esentation de $\CC^{\times} \times \SL (2, \CC)$, on a donc
$$\psi = \chi \oplus \chi^{-1} \oplus \sum_i \chi_i \otimes r_i , $$
o\`u chaque $\chi_i = z^{p_i} \overline{z} \- {}^{q_i}$ avec $p_i + \frac{n_i+1}{2} \in \ZZ$. 
Les d\'emonstrations \cite[p. 78-80 \& p. 82-83]{BC} impliquent finalement une contradiction.
\end{proof}

\subsection{} Pour tout $\tau$ fix\'e et dans le cas (2) de la proposition, la longueur du param\`etre 
$s \in (n-1)/2 + \ZZ$ est contr\^ol\'ee par la description, connue, des s\'eries compl\'ementaires. 
Pour $0 \leq k \leq n-1$, notons $\tau_k$ la repr\'esentation standard de  ${}^0 \- M=\SO (n-1)$ sur 
$\wedge^k \CC^{n-1}$. Les repr\'esentations $\tau_k$ et $\tau_{n-1-k}$ sont \'equivalentes, 
$\tau_k$ est irr\'eductible pour $k \neq (n-1)/2$, et $\tau_{(n-1)/2} = \tau_{(n-1)/2}^+ \oplus 
\tau_{(n-1)/2}^-$ somme directe de deux repr\'esentations irr\'eductibles.
Ces repr\'esentations sont les seules repr\'esentations ${}^0 \- M$ 
apparaissant dans $\wedge^* \mathfrak{p}_{\CC}$, o\`u $\mathfrak{p}_{\CC} = \CC^n$. 
On a plus pr\'ecisemment les d\'ecompositions~:
$$\left( \wedge^k \CC^n \right)_{| \SO (n-1)} = \left\{
\begin{array}{ll}
\tau_k \oplus \tau_{k-1} & \mbox{ si } 1 \leq k < \frac{n-1}{2} , \\
\tau_k^+ \oplus \tau_k^- \oplus \tau_{k-1} & \mbox{ si } k = \frac{n-1}{2} , \\
\tau_{k} \oplus \tau_k & \mbox{ si } k = \frac{n}{2} . 
\end{array} \right. $$
{\it Via} la formule de Matsushima les repr\'esentations $J(\tau_k^{\pm} , s)$ correspondent aux 
$k$-formes diff\'erentielles coferm\'ees. On a par ailleurs 
$$
\begin{array}{ll}
\lambda_{\tau_k} = \left( \frac{n-1}{2}  , \ldots , \widehat{\frac{n-1}{2}-k} , \ldots , \frac{n-1}{2} -\ell \right)  & \mbox{ si } k \neq \frac{n-1}{2} , \\
\lambda_{\tau_{k}^{\pm}} = \left( \frac{n-1}{2}  , \ldots ,  2 , \pm 1\right) & \mbox{ si } k = \frac{n-1}{2}. 
\end{array}
$$
La repr\'esentation $J(\tau_k , s)$ est unitaire si $s \in i \RR$ ou $s \in \RR$ et $|s |\leq (n-1)/2 -k$; pour
les repr\'esentations d'Arthur  -- consid\'er\'ees dans la proposition \ref{PA} -- on a donc $s \in i \RR$ 
ou $s \in \RR$ avec soit $|s| < 1/2 - 1/(N^2+1)$ soit $\pm s \in \left\{ \frac{n-1}{2} - \ell , \ldots , \frac{n-1}{2} - k \right\}$.  

\begin{thm} \label{Tprinc}
Soit $k$ un entier compris entre $0$ et $\ell$ et soit $\Gamma \subset G$ un sous-groupe
de congruence. Alors, le spectre du laplacien sur les $k$-formes diff\'erentielles $L^2$ et
coferm\'es de $\Gamma \backslash \H^n$ est contenu dans l'ensemble 
$$\bigcup_{k \leq j \leq \ell} \left\{ \left( \frac{n-1}{2}-k \right)^2 - \left( \frac{n-1}{2} - j \right)^2 \right\}  \cup 
\left[ \left( \frac{n-1}{2} -k \right)^2 - \left( \frac12 - \frac{1}{N^2 +1} \right)^2 , + \infty \right[ .$$
En particulier, la conjecture $A^-$ de \cite{BC} est toujours v\'erif\'ee et pour tout $k \leq (n-4)/2$, 
la premi\`ere valeur propre non nulle $\lambda_1^k = \lambda_1^k (\Gamma)$ v\'erifie~:
$$\lambda_1^k \geq n-2k-2 .$$
\end{thm}

\subsection{}
Revenons au cas g\'en\'eral o\`u $G_{v_0} = \SO (p,q)$. Notons $\theta$ l'involution de Cartan associ\'ee
au choix de ${\rm S} ({\rm O} (p) \times {\rm O} (q))$ comme sous-groupe compact maximal $K$ dans $G_{v_0}$. Nous nous int\'eressons \`a l'isolation des rep\'esentations unitaires $\sigma \in 
\widehat{G}_{v_0}$ dont la $(\mathfrak{g}_{\CC} , K)$-cohomologie 
$$H^{\bullet} (\mathfrak{g}_{\CC} , K ; \sigma ) = \bigoplus_k H^k (\mathfrak{g}_{\CC} , K ; \sigma )$$
est non nulle.

Les param\`etres locaux $\tilde{\psi}_{v_0} : W_{\RR} \times \SL (2, \CC) \rightarrow {}^L \- G$ associ\'es aux repr\'esentations cohomologiques sont d\'ecrits par Arthur \cite{ArthurOU}. Rappelons d'abord que les 
repr\'esentations cohomologiques sont classifi\'ees par Vogan et Zuckerman \cite{VoganZuckerman},
ce sont les modules $A_{\mathfrak{q}}$ associ\'ees \`a des sous-alg\`ebres paraboliques $\theta$-stables 
$\mathfrak{q} = \mathfrak{l} + \mathfrak{u} \subset \mathfrak{g}_{\CC}$. 
Notons $L$ le sous-groupe de Levi correspondant \`a $\mathfrak{l}$. Alors $L$ est d\'efini sur $\RR$. On peut identifier
son groupe dual $\widehat{L}$ au sous-groupe de Levi correspondant dans $\widehat{G}$.
D'apr\`es Shelstad \cite{Shelstad}, l'injection 
$\widehat{L} \subset \widehat{G}$ s'\'etend en un plongement canonique 
$\xi_L : {}^L \- L \rightarrow {}^L \- G$ de $L$-groupes.
Notons $\psi_L : W_{\RR} \times \SL (2 , \CC) \rightarrow {}^L \- L$ le param\`etre 
qui est trivial en restriction \`a $W_{\RR}$ et envoie l'\'el\'ement 
$$\left( 
\begin{array}{cc}
1 & 1 \\
0 & 1 
\end{array} \right) \in \SL (2 , \CC)$$
sur l'\'el\'ement unipotent principal dans $\widehat{L}$. Le param\`etre $\psi_L$ 
correspond \`a la repr\'esentation triviale de $L$ et le param\`etre compos\'e $\xi_L \circ \psi_L$
est le param\`etre local associ\'e \`a la repr\'esentation
cohomologique correspondant \`a la sous-alg\`ebre parabolique $\mathfrak{q}$.

\begin{thm} \label{Topq}
Soit $\sigma = A_{\mathfrak{q}}$ une repr\'esentation cohomologique de $G_{v_0}= \SO (p,q)$ telle que 
le centre
de $L$ soit compact.
Alors $\sigma$ est isol\'ee de
$$\left\{ \pi \in \widehat{G}_{v_0} \; : \;  \mbox{ il existe } \Pi \mbox{ repr\'esentation automorphe de } G \mbox{ telle que } \Pi_{v_0} \cong \pi \right\}$$
dans $\widehat{G}_{v_0}$ -- le dual unitaire -- muni de sa topologie de Fell. 
C'est en particulier toujours le cas lorsque 
${\rm rang}_{\CC} (G) = {\rm rang}_{\CC} (K)$, {\it i.e.} lorsque $pq$ est pair.
\end{thm}
\begin{proof}[D\'emonstration] Soit $L$ le sous-groupe de Levi associ\'e \`a une certaine sous-alg\`ebre
parabolique $\mathfrak{q}$. Le groupe $L$ est donn\'e naturellement par des matrices diagonales par blocs (cf. 
la description du cas unitaire dans \cite[\S 5.2]{BC}). Il se d\'ecompose ainsi en facteurs $L_i$, 
simples ou ab\'eliens. Il en est naturellement de m\^eme pour le
groupe dual. Le param\`etre $\psi_L$ se d\'ecompose donc en une somme directe de param\`etre
$\psi_{L_i}$. 

Si $L_i$ est un facteur simple localement isomorphe \`a $\SO (n,1)$ ou $\SU(n,1)$ $(n \geq 1)$
alors la repr\'esentation
triviale de $L_i$ n'est pas isol\'ee dans le dual unitaire de $L_i$~: si $L_i$ est localement isomorphe \`a $\SO (1,1)$ il suffit de consid\'erer une suite de caract\`eres unitaires
convergeant vers le caract\`ere trivial, sinon $1_{L_i}$ est limite de repr\'esentations de la 
s\'erie compl\'ementaire unitaire de $\SO (n,1)$ ($n \geq 2$) ou $\SU (n,1)$ ($n\geq 1$). Dans tous les cas notons $\psi_{i}^{s} :  \CC^{\times} \times \SL (2 , \CC)  \rightarrow  {}^L \- L_i$ le param\`etre
$$\left[ \left(
\begin{array}{cc}
(z/\overline{z})^s & \\
& (z/\overline{z})^{-s} 
\end{array} \right) \otimes 1\right] \oplus \left[ 1 \otimes r_{n-1} \right] ,$$
o\`u $r_{n-1}$ d\'esigne la repr\'esentation irr\'eductible de $\SL (2, \CC)$ qui est de
dimension maximale dans le groupe dual de $\SO (n-1)$, resp. $\SU (n-1)$. Lorsque $s = (n-1)/2$,
dans le cas orthogonal, et $s = n/2$, dans le cas unitaire, le caract\`ere infinit\'esimal de ce param\`etre est celui de la repr\'esentation triviale de $L_i$. 

Pour qu'un param\`etre $\psi_i^s$ soit un param\`etre d'Arthur g\'en\'eralis\'e il est n\'ecessaire que 
$|{\rm Re} (s) |< 1/2$. Les caract\`eres infinit\'esimaux de tels param\`etres ne peuvent donc s'approcher
de celui de la repr\'esentation triviale que si $n=1$, c'est-\`a-dire si $L_i$ est localement isomorphe 
\`a $\SO (1,1)$. Dans ce cas le centre de $L$ n'est pas compact.  

Comme remarqu\'e dans \cite[Remarque (5) p. 217]{JIMJ}, il d\'ecoule de
\cite{Vogan} que les repr\'esentations irr\'eductibles unitaires qui s'approchent  de repr\'esentations
cohomologiques de $G_{v_0}$ sont des induites cohomologiques. D'apr\`es le th\'eor\`eme \ref{C},
si une telle repr\'esentation est la composante locale d'une repr\'esentation automorphe, son 
caract\`ere infinit\'esimal est \'egal \`a celui d'un param\`etre $\bigoplus_i \psi_i^s$ comme ci-dessus.
Cela force les repr\'esentations \`a \^etre induites d'un sous-groupe de Levi $L$ \`a centre
non compact. Auquel cas la repr\'esentation cohomologique limite est elle-m\^eme associ\'ee \`a 
une sous-alg\`ebre parabolique dont le Levi a un centre non compact.
\end{proof}

\bibliography{bibli}

\bibliographystyle{smfplain}

\end{document}